\documentclass{amsart}

\usepackage{macros}
\standardsettings

\newcommand{\concat}{*}
\newcommand{\qred}{q_{\mathrm{red}}}
\newcommand{\qint}{q_{\mathrm{int}}}
\newcommand{\qmax}{q_{\mathrm{max}}}
\newcommand{\pseud}{\ell_{\mathrm{p}}}
\DeclareMathOperator{\denom}{denom}

\draftfalse

\begin{document}
\title[Intrinsic approximation for fractals]{Intrinsic approximation for fractals defined by rational iterated function systems - Mahler's research suggestion}

\authorlior\authordavid


\maketitle

\begin{abstract}

In this paper, we consider intrinsic Diophantine approximation in the sense of K. Mahler (1984) on the Cantor set and similar fractals. We begin by obtaining a Dirichlet type theorem for the limit set of a rational iterated function system. Next, we investigate the rigidity of this result by applying a random affine transformation to such a fractal and determining the intrinsic Diophantine theory of the image fractal. The final two sections concern the optimality of the Dirichlet type theorem established at the beginning. The first of these seeks to show optimality in the sense that any proof using the same method as ours cannot prove a better approximation exponent, in a precise sense. This is done by introducing a new height function on the rationals intrinsic to the fractal and studying the Diophantine properties of points on the fractal with respect to this new height function. In the final section, we use a result of S. Ramanujan to give a lower bound on the periods of rationals which could cause exceptions to the optimality of the approximation exponent (this time with the usual height function). We give a heuristic argument suggesting that there are only finitely many rationals with periods so large; if this is true, then the approximation exponent is optimal for the Cantor set.

\end{abstract}

\section{Introduction}

In 1984, K. Mahler published a paper entitled ``Some suggestions for further research'' \cite{Mahler}, in which he writes the following moving statement: ``At the age of 80 I cannot expect to do much more mathematics. I may however state a number of questions where perhaps further research might lead to interesting results''. One of these questions was regarding intrinsic and extrinsic approximation on the Cantor set.\footnote{In this paper, the phrase ``Cantor set'' always refers to the ternary Cantor set.} In Mahler's words, ``How close can irrational elements of Cantor's set be approximated by rational numbers

\begin{enumerate}
\item	In Cantor's set, and
\item	By rational numbers not in Cantor's set?''\footnote{Our paper is mainly concerned with the first question; we will consider the second in \cite{FishmanSimmons2}.}

\end{enumerate}

In contrast to intrinsic approximation on the Cantor set in particular and on fractals in general, the Diophantine approximation theory of the real line is classical, extensive, and essentially complete as far as characterizing how well real numbers can be approximated by rationals (\cite{Schmidt3} is a standard reference). The basic result on approximability of all reals is

\begin{theorem*}[Dirichlet's Approximation Theorem] 
For each $x \in \mathbb{R}$ and for any $Q \in \mathbb{N}$ there exists $p/q\in \mathbb{Q}$ with $1\leq q \leq Q$, such that
\[
\bigl|x-p/q\bigr|<\frac{1}{qQ}.
\]
\end{theorem*}

\begin{corollary*} For every irrational $x \in \mathbb{R}$, 
\[
\label{dirich}
\bigl|x-p/q\bigr|<\frac{1}{q^2}
\]
for infinitely many $p/q\in\mathbb{Q}$.
\end{corollary*}

The optimality of this approximation function (up to a multiplicative constant) is demonstrated by the existence of badly approximable numbers, i.e. reals $x$ such that for some $c(x)>0$
\[
\bigl|x- p/q \bigr| > \frac{c(x)}{q^2} \hspace{2mm} \text{ for all } p/q \in \mathbb{Q}.
\]

It is well known that the set of very well approximable numbers, i.e. the set of all reals $x$ satisfying for some positive $\epsilon(x)$
\[
\bigl|x- p/q \bigr| < \frac{1}{q^{2+\epsilon(x)}}
\]
for infinitely many rationals $p/q$, is null. This fact demonstrates that this approximation function cannot be improved for almost all irrationals. We remark that the subject of approximating points on fractals by rationals has been extensively studied in recent years; see for example \cite{Fishman, KleinbockWeiss1, KTV} for badly approximable numbers and \cite{KLW, Weiss} for very well approximable numbers. In \cite{Bugeaud2}, elements of the Cantor set with any prescribed irrationality exponent were explicitly constructed.

Recently in \cite{BFR}, R. Broderick, A. Reich, and the first named author made what could be considered as a first step towards answering Mahler's question:

\begin{proposition}[{\cite[Corollary 2.2]{BFR}}]
\label{BFRresult}
Let $C$ be the Cantor set and $d = \dim C$. Then 
for all $x\in C$, there exist infinitely many solutions $p\in\mathbb{Z}$,
$q\in\mathbb{N}$, $p/q\in C$ to \[
\bigl|x-p/q\bigr|<\frac{1}{q(\log _{3} q)^{1/d}}.\]
\end{proposition}

The proof of the above proposition crucially depends on the $\times 3$ invariance of the middle third Cantor set, and a similar result was proven in \cite{BFR} for any $\times d$-invariant totally disconnected Cantor-like set.

The main motivation of this paper is to provide a better understanding of intrinsic Diophantine approximation on fractals. We do this by first generalizing Proposition \ref{BFRresult} by removing the $\times d$ constraint (Theorem \ref{theoremdirichlet}):

\begin{definition}
Let $J$ be a subset of $\R$ and let $\psi:\N\rightarrow(0,\infty)$ be any function. We will say that a point $x\in J$ is \emph{intrinsically approximable with respect to $\psi$} if there exist infinitely many rationals $p/q\in\Q\cap J$ such that
\[
\left|x - p/q\right| \leq \frac{\psi(q)}{q}.
\]
We will say that $x$ is \emph{badly intrinsically approximable with respect to $\psi$} if there exists $\varepsilon > 0$ such that $x$ is not intrinsically approximable with respect to the function $\varepsilon\psi$. Otherwise, we will say that $x$ is \emph{intrinsically well approximable with respect to $\psi$}.
\end{definition}

\begin{definition}
\label{definitionIFS}
Let $E$ be a finite set. An \emph{iterated function system} (\emph{IFS}) on $\R$ is a collection $(u_a)_{a\in E}$ of contracting similarities $u_a:\R\to\R$ satisfying the \emph{open set condition}: there exists an open set $W\subset\R$ such that the collection $(u_a(W))_{a\in E}$ is a disjoint collection of subsets of $W$ (see \cite{Hutchinson} for a thorough discussion). The \emph{limit set} of $(u_a)_{a\in E}$ is the image of the \emph{coding map} $\pi:E^\N\to\R$ defined by
\[
\pi(\omega) = \lim_{n\to\infty} u_{\omega_1}\cdots u_{\omega_n}(0),
\]
and will be denoted $J$. We will call the IFS $(u_a)_{a\in E}$ \emph{rational} if for each $a\in E$, $u_a$ preserves $\Q$, i.e.
\begin{equation}
\label{rationalIFS}
u_a(x) = \frac{p_a}{q_a}x + \frac{r_a}{q_a}\hspace{5mm}\text{with }\hspace{5mm} p_a,r_a\in\Z, q_a\in\N.
\end{equation}
\end{definition}

\begin{reptheorem}{theoremdirichlet}
Suppose that $(u_a)_{a\in E}$ is a rational IFS and let $J$ be the limit set of this IFS. Let $\delta$ denote the Hausdorff dimension of $J$. Let
\[
\label{gamma}
\gamma := \max_{a\in E} \frac{\log|p_a|}{\log(q_a)},
\]
where $p_a,q_a$ are as in \textup{(\ref{rationalIFS})}. There exists $K < \infty$ such that for each $x\in J$ and for each $Q\geq \qmax := \max_a q_a$ there exists $p/q\in \Q\cap J$ with $q\leq Q$ such that
\[
\left|x - p/q\right| \leq K q^{\gamma - 1}\log(Q)^{-1/\delta}.
\]
In particular, if $x$ is irrational then $x$ is intrinsically approximable with respect to the function
\[
\label{qgammalnqdelta}
\psi(q) := K q^\gamma \log(q)^{-1/\delta}.
\]
\end{reptheorem}

Notice that the Dirichlet-type theorems in \cite{BFR} are immediate consequences of Theorem \ref{theoremdirichlet} as $\gamma = 0$ whenever $p_a = \pm 1$ for all $a\in E$.\\

Let $J$ be a $\times d$-invariant limit set, i.e. the set of all points $x\in [0,1]$ such that the digits of the base $d$ expansion of $x$ are contained in some fixed set $E\subset\{0,\ldots,d - 1\}$ (e.g. the Cantor set). The algebraic structure makes it easy to find rationals in $J$, but in Section \ref{translations} we consider what happens when we lose (most of) the algebraic structure and keep only the geometric structure when considering random translations and dilations. Specifically, we wish to investigate how ``generic'' the approximation function in Theorem \ref{theoremdirichlet} is with respect to random translations and dilations.

We begin by showing that if $f$ is an affine transformation, then it is unlikely that $f(J)$ intersects $\Q$ densely in the following senses:
\begin{itemize}
\item If $f$ is random, then $f(J)$ does not intersect $\Q$.
\item If $f$ is random subject to $f(J)\cap\Q\neq\emptyset$, then $\#(f(J)\cap\Q) = 1$.
\item If $f$ is random subject to $\#(f(J)\cap\Q)\geq 2$, then $\#(f(J)\cap\Q) = 2$.
\end{itemize}

We remark that the third of these assertions is the only one which is difficult to prove.\comlior{changed the phrasing slightly in the last 2 paragraphs} Formally, we have the following:

\begin{repobservation}{observationrandomsimilarity}
Let $\lambda$ be Lebesgue measure on $\R$, and define $\mu = \HH^\delta\given J$.
For each $x,y\in\R$, let $f_{x,y}$ be the unique simiarity so that $f_{x,y}(x) = 0$ and $f_{x,y}(y) = 1$. Then:
\begin{itemize}
\item[(i)] For $\lambda\times\lambda$-almost every $(x,y)\in\R^2$, we have $f_{x,y}(J)\cap\Q = \emptyset$.
\item[(ii)] For $\mu\times\lambda$-almost every $(x,y)\in\R^2$, we have $f_{x,y}(J)\cap\Q = \{0\}$.
\end{itemize}
In particular, in each of these cases $f_{x,y}(J)$ does not intersect the rational numbers densely.
\end{repobservation}

\begin{reptheorem}{theoremrandomsimilarity}
Let $J$ be a $\times d$-invariant set. Then for $\mu\times\mu$-almost every $x,y\in J$,
\[
f_{x,y}(J)\cap\Q = \{0,1\}.
\]
In particular, $f_{x,y}(J)$ does not intersect densely with $\Q$.
\end{reptheorem}

Next, we consider the case where $f$ is not just an affine transformation but a translation. In this case, if $f$ is random, then $f(J)$ still does not intersect $\Q$ (Observation \ref{observationrandomtranslation}), but if $f(J)$ does intersect $\Q$, then it does so densely. In fact, we can prove more:

\begin{replemma}{dense}
Suppose that $J$ is a $\times d$-invariant limit set. For each $x\in\R$, let $f_x(t) = t - x$, and let $J_x = f_x(J)$. Suppose that $x\in\R$ is such that $\Q\cap J_x\neq\emptyset$. Then $\Q\cap J_x$ is dense in $J_x$. Furthermore, there exists $K < \infty$ such that every point $y\in J_x$ is intrinsically approximable with respect to the constant function
\[
\label{Kqq}
\psi(q) = K.
\]
\end{replemma}

A natural question is whether the approximation function is optimal. In order to answer this question we will assume that
\begin{enumerate}[i)]
\item $d$ is prime, and
\item the set $E$ of allowable digits contains no two adjacent integers, and contains both $0$ and $(d - 1)$.
\end{enumerate}
These assumptions are satisfied, for example, if $J$ is the Cantor set. As before, let $\delta$ be the Hausdorff dimenson of $J$, and let $\mu = \HH^\delta\given J$.

\begin{reptheorem}{theoremcase0ba}
Suppose that $J$ satisfies \textup{(i)-(ii)}, and let $x\in J$ be a $\mu$-random point. Then the set of numbers $y\in J_x$ which are badly intrinsically approximable with respect to the approximation function $\psi(q)=K$ is of Hausdorff dimension $\delta$.
\end{reptheorem}

\ignore{
We will consider the following case:
\begin{enumerate}[i)]
\item $d$ is prime
\item The set $E$ of allowable digits contains no two adjacent integers, and contains both $0$ and $(d - 1)$.
\item Every rational in $J_x$ is of the form $p/d^n$ for some $p,n\in\N$.
\end{enumerate}
If (i)-(iii) hold, we shall say that $J$ and $x$ satisfy Case 0.
}

In addition, we demonstrate a correspondence between the above case and the case considered by Levesley, Salp, and Velani \cite{LSV}, who approximated points in the Cantor set by the left endpoints of the Cantor set. By translating their results into our setting we can prove the following theorem:

\begin{reptheorem}{theoremcase0khinchin}
Let $J$ satisfy \textup{(i)-(ii)}, and let $x\in J$ be a $\mu$-random point. Fix any function $\psi:\N\rightarrow(0,\infty)$. Let $f$ be a dimension function such that $t\mapsto t^{-\delta}f(t)$ is monotonic. If we denote the set of $\psi$-intrinsically well approximable points by $\WA_{\psi,\mathrm{int}}$, then
\[
\HH^f(\WA_{\psi,\mathrm{int}}) = \begin{cases}
0 & \textup{if }\sum_{n = 1}^\infty f(\psi(d^n)/d^n) (d^n)^\delta < \infty\\
\HH^f(J) & \textup{if }\sum_{n = 1}^\infty f(\psi(d^n)/d^n) (d^n)^\delta = \infty
\end{cases},
\]
where $\HH^f$ is Hausdorff $f$-measure.
\end{reptheorem}

\begin{repcorollary}{corKhinchin}
If $\psi(q) = \log(q)^{-1/\delta}$, then almost every point is intrinsically well approximable with respect to $\psi$; if $\psi(q) = \log(q)^{-(1 + \varepsilon)/\delta}$, then almost every point is badly intrinsically approximable with respect to $\psi$.
\end{repcorollary}

\ignore{
Additionally, we prove the following theorem demonstrating the optimality of the Dirichlet approximation function $\psi(q) = 1$:

\begin{reptheorem}{theoremcase0ba}
Let $J$ and $x$ satisfy Case 0 and let $\delta$ be the Hausdorff dimension of $J$. Then the set of numbers $y\in J_x$ which are badly intrinsically approximable with respect to the approximation function $\psi(q)=K$ is of Hausdorff dimension $\delta$.
\end{reptheorem}
}


\ignore{
In fact, Case 0 is the only case in which we are able to say anything about the optimality of the approximation function. We remark that (i)-(ii) are reasonable and easy-to-check assumptions that are satisfied e.g. for the standard Cantor set. Finally, we prove that condition (iii) is generic:

\begin{reptheorem}{exceptions}
Suppose that $J$ is a $\times d$-invariant limit set satisfying \textup{(i)-(ii)}. The set $S$ consisting of all $x\in\bigcup_{p/q\in\Q}J_{p/q}$ for which \textup{(iii)} does NOT hold is small with respect to both measure and category, i.e. $\HH^\delta(S) = 0$ and $S\cap J$ is meager in $J$.
\end{reptheorem}
}

In Sections \ref{sectionintrinsicdenominator} and \ref{optimality} we discuss the question of whether the approximation function of Theorem \ref{theoremdirichlet} is optimal. The starting point is the observation that the method of Theorem \ref{theoremdirichlet} produces only rational numbers of a particular form. Specifically, if we let $\pi:E^\N\rightarrow J$ be the coding map (see Definition \ref{definitionIFS} above), then Theorem \ref{theoremdirichlet} produces rationals of the form $\pi(\omega)$, where $\omega\in E^\N$ is an eventually periodic word.

Secondly, when Theorem \ref{theoremdirichlet} does produce a rational number then it does not produce it in reduced form. For example, the fraction $1/4$ in the Cantor set $C$ would be represented as $2/8$. Consequently, we will call the number $8$ the \emph{intrinsic denominator} of $1/4$ with respect to the fractal $C$ (defined precisely in Section \ref{sectionintrinsicdenominator}). Note that a rational can only have an intrinsic denominator if it has a preimage which is an eventually periodic word. We denote the intrinsic denominator of $p/q$ given by the IFS by $\qint$, whereas the denominator of $p/q$ in reduced form will be denoted $\qred$. (Following our above example, $\qint =8$ while $\qred =4$). It is easily observed that for every $p/q\in J$ we have $\qred \divides \qint$.

As a result of this analysis, the question of whether Theorem \ref{theoremdirichlet} is optimal can now be split into three sub-questions:

\begin{enumerate}[(i)]
\item Does every rational in $J$ have a preimage under $\pi$ which is eventually periodic? In other words, are the rationals in $J$ that have intrinsic denominators the only ones that exist?
\item If $p/q$ is a rational in $J$ that has an intrinsic denominator, what is the ratio between its intrinsic denominator $\qint$ and its reduced denominator $\qred$?
\item Is the approximation function (\ref{qgammalnqdelta}) optimal if we only consider rationals in $J$ which come from periodic words, and if we consider the intrinsic denominator to be the true denominator of the rational?
\end{enumerate}

In Section \ref{sectionintrinsicdenominator} we will consider questions (i) and (iii), and in Section \ref{optimality} we will consider question (ii). In each case we have only partial results. In appears that all three questions are hard when considered in full generality, although it seems (ii) is the hardest.

Question (i) is the easiest to deal with. In the case of the Cantor set (or more generally of a $\times d$-invariant set), the answer is already well-known. The fact that every rational in $J$ has a preimage under $\pi$ which is eventually periodic is merely a restatement of the fact that every rational number has an eventually periodic base $d$ expansion. We slightly generalize this result with the following lemma:

\begin{replemma}{lemmacanonical}
Suppose that $p_a = \pm 1$ for all $a\in E$, where $p_a$ are given by \textup{(\ref{rationalIFS})}. Then every rational in $J$ is the image of an eventually periodic word (and therefore has an intrinsic denominator).
\end{replemma}

We next consider question (iii):
\begin{repdefinition}{badsymbolic}
Let $\psi:(0,\infty)\rightarrow(0,\infty)$ be a nonincreasing function. A point $x\in J$ is said to be \emph{badly symbolically approximable with respect to $\psi$} if there exists $\varepsilon > 0$ such that for all $p/q\in\Q\cap J$ we have
\[
\left|x - p/q\right| \geq \varepsilon\frac{\psi(\qint)}{\qint}.
\]
Otherwise, $x$ is said to be \emph{symbolically well approximable with respect to $\psi$}.
\end{repdefinition}

It thus follows that badly intrinsically approximable implies badly symbolically approximable, but not vice-versa.

Rather than attempting to demonstrate the existence of numbers which are badly symbolically approximable with respect to the Dirichlet function (\ref{qgammalnqdelta}), we instead prove a Khinchin-type theorem. Our motivation for this is that it seems less likely that the intrinsic denominator differs greatly from the denominator in reduced form for the rational approximations of almost every point, than that it differs greatly for the approximants of a single point.

An immediate corollary of Theorem \ref{theoremkhinchin} is the following:

\begin{repcorollary}{KhinchinCantor}
Let $C$ be the Cantor set and $\mu$ the Hausdorff measure in the Cantor's set dimension restricted to $C$.
Then for $\mu$-almost every $x\in J$, $x$ is badly symbolically approximable with respect to $\psi(q) = \log(q)^{-(2/\delta + \varepsilon)}$ 
and is symbolically well approximable with respect to $\psi(q) = \log(q)^{-2/\delta}.$
\end{repcorollary}

In Section \ref{optimality}, we restrict ourself to the case where the limit set is the Cantor set $C$. We begin by recalling the following conjecture from \cite{BFR}:

\begin{repconjecture}{conjectureBFR}[{\cite[Conjecture 3.3]{BFR}}]
If
\[
S_n := \{p/q\in C:\gcd(p,q) = 1,\;3^{n - 1} \leq q < 3^n\}
\]
then for all $\varepsilon_1 > 0$ we have
\[
\#(S_n) = O(2^{n(1 + \varepsilon_1)}).
\]
\end{repconjecture}

This conjecture is immediately relevant to intrinsic approximation as it implies (see \cite[Corollary 3.4]{BFR}) that $\mu(\VWA_C) = 0$, where
\[
\VWA_C := \{x\in C:\exists \varepsilon > 0\;\;\exists^\infty p/q\in C \;\; |x - p/q| \leq q^{-(1 + \varepsilon)}\}.
\]
In particular, this would imply that $C\butnot\VWA_C\neq\emptyset$ and so the approximation exponent is optimal in Theorem \ref{theoremdirichlet}.

We cannot prove Conjecture \ref{conjectureBFR} at this time, but we will reduce it to a simpler conjecture which a heuristic argument suggests is true. Suppose that $p/q$ is a rational number. The \emph{period} of $p/q$ is the period of the ternary expansion of $p/q$, and will be denoted $P(p/q)$.

The first step we make is proving the following theorem (using a result of Ramanujan \cite{Ramanujan} concerning the number-of-divisors function):

\begin{reptheorem}{theoremSnK}
For every $K < \infty$, if
\[
S_n^{(K)} := \{p/q\in C:\gcd(p,q) = 1,\;3^{n - 1} \leq q < 3^n,\textup{ and }P(p/q) \leq K \log(q)\}
\]
then for all $\varepsilon_1 > 0$ we have
\[
\#(S_n^{(K)}) = O(2^{n(1 + \varepsilon_1)}).
\]
\end{reptheorem}

We then provide a heuristic argument to support the following conjecture:

\begin{repconjecture}{conjectureheuristic}
For all $K > 2/\log(3/2)$, we have $S_n^{(K)} = S_n$ for all $n$ sufficiently large. In particular
\[
\#(S_n\butnot S_n^{(K)}) = o(1).
\]
\end{repconjecture}

The following is a corollary of Theorem \ref{theoremSnK}:

\begin{repcorollary}{corollaryoftheoremSnK}
Conjecture \ref{conjectureheuristic} implies Conjecture \ref{conjectureBFR}, implying $\mu(\VWA_C) = 0$.
\end{repcorollary}

\ignore{

Finally, let us note that Conjecture \ref{conjectureBFR} also has relevance to Mahler's second question regarding extrinsic approximation. Indeed, if $x$ is any point which is badly intrinsically approximable with respect to the function
\[
\psi(q) = q^{-1 + \varepsilon},
\]
for some $\varepsilon > 0$, then $x$ is extrinsically approximable with respect to Dirichlet's function $\psi(q) = q^{-1}$. To see this, note that by Dirichlet's Theorem there is a sequence $p_n/q_n\tendsto n x$ such that $|x - p_n/q_n|\leq 1/q_n^2$, but since $x$ is badly intrinsically approximable with respect to $\psi(q) = q^{-1 + \varepsilon}$, it follows that only finitely many of these approximations can be intrinsic. Thus if Conjecture \ref{conjectureBFR} is correct, then almost every point on the Cantor set is extrinsically approximable with respect to the function $\psi(q) = q^{-1}$.\\

}

\textbf{Acknowledgements.} Both authors would like to thank Y. Bugeaud and M. Urba\'nski for helpful suggestions and comments.
This work was partially supported by a grant from the Simons Foundation \#245708.

\newpage

\section{A Dirichlet-type theorem for fractals}
\label{general}

We consider a finite set (alphabet) $E$ and denote by $E^r$ the set of all words of length $r$ formed using this alphabet and by $E^*$ the set of all words, finite or infinite, formed using this alphabet. If $\omega\in E^*$, then we denote subwords of $\omega$ by
\[
\omega_{n + 1}^{n + r} := (\omega_{n + i})_{i = 1}^r \in E^r.
\]
We denote the concatenation of $\omega$ and $\tau$ by $\omega\concat\tau$. Furthermore, we define the shift map
\[
\sigma:E^\N\rightarrow E^\N
\]
by $\sigma(\omega) = \omega_2^\infty = (\omega_{i + 1})_{i\in\N}$. If $\omega\in E^n$ is a finite word then we define
\[
u_\omega(x) := u_{\omega_1}\circ\ldots\circ u_{\omega_n}(x).
\]
We define the map $\pi:E^\N\rightarrow \R$ by
\[
\pi(\omega) := \lim_{n\rightarrow\infty}u_{\omega_1^n}(0)
\]
and we define the limit set $J$ to be the image of this map. Let $\delta$ be the Hausdorff dimension of $J$, and let $\mu$ be the $\delta$-dimensional Hausdorff measure restricted to $J$, normalized to be a probability measure.

\begin{theorem}[Dirichlet for fractals]
\label{theoremdirichlet}
Suppose that $(u_a)_{a\in E}$ is a rational IFS and let $J$ be the limit set of this IFS. Let
\begin{equation}
\label{gamma}
\gamma := \max_{a\in E} \frac{\log|p_a|}{\log(q_a)},
\end{equation}
where $p_a,q_a$ are as in \textup{(\ref{rationalIFS})}. There exists $K < \infty$ such that for each $x\in J$ and for each $Q\geq \qmax := \max_a q_a$ there exists $p/q\in \Q\cap J$ with $q\leq Q$ such that
\[
\left|x - p/q\right| \leq K q^{\gamma - 1}\log(Q)^{-1/\delta}.
\]
In particular, if $x$ is irrational then $x$ is intrinsically approximable with respect to the function
\begin{equation}
\label{qgammalnqdelta}
\psi(q) := K q^\gamma \log(q)^{-1/\delta}.
\end{equation}
\end{theorem}

\begin{proof}
Recall [\cite[Theorem 4.14]{Mattila}] that the measure $\mu := \HH^\delta\given_J/\HH^\delta(J)$ is Ahlfors $\delta$-regular on $J$, i.e.
\[
\mu(B(x,r))\asymp r^\delta,
\]
where $x\in J$ and $0 < r \leq 1$.\footnote{Here and from now on the symbols $\lesssim$, $\gtrsim$, and $\asymp$ denote multiplicative asymptotics; an addition of the subscript $\plus$ denotes an additive asymptotic. For example, $A\lesssim_\plus B$ means that there exists a constant $C > 0$ such that $A\leq B + C$.} Thus if $(x_n)_{n = 0}^{N - 1}$ is an $r$-separated sequence in $J$, i.e. if $\dist(x_n,x_m)\geq r$ for all $n\neq m$, then the balls $(B(x_n,r/2))_{n = 0}^{N - 1}$ are disjoint and so
\[
1 = \mu(J) \geq \sum_{n = 0}^{N - 1}\mu(B(x_n,r/2)) \asymp \sum_{n = 0}^{N - 1}(r/2)^\delta \asymp N r^\delta.
\]
Thus there exists $K_1 < \infty$ depending only on $J$ such that $N r^\delta\leq K_1$. Taking the contrapositive gives
\begin{lemma}[Fractal pigeonhole principle]
If $(x_n)_{n = 0}^N$ is any finite sequence in $J$, then there exist distinct integers $0\leq n,m\leq N$ such that
\[
|x_n - x_{n + m}| \leq r_N := (N/K_1)^{-1/\delta}.
\]
\end{lemma}

Now suppose we have $x\in J$ and $Q \geq \qmax$. Fix $N\in\N$ to be determined. Let $\omega\in E^\N$ be a preimage of $x$ under $\pi$. We consider the iterates of $\omega$ under the shift map $\sigma$. We apply the fractal pigeonhole principle to the sequence $(\pi\circ\sigma^n(\omega))_{n = 0}^N$ to conclude that there exist two integers $0 \leq n < n + m \leq N$ such that if $y = \pi\circ\sigma^n(\omega)$ and $z = \pi\circ\sigma^{n + m}(\omega)$, then $|y - z| \leq r_N$.

Let $u_{(1)} = u_{\omega_1^n}$ and let $u_{(2)} = u_{\omega_{n + 1}^{n + m}}$. Then $x = u_{(1)}(y)$, and $y = u_{(2)}(z)$. Let
\begin{align*}
p_{(1)} &= p_{\omega_1}\cdots p_{\omega_n}&
p_{(2)} &= p_{\omega_{n + 1}}\cdots p_{\omega_{n + m}}\\
q_{(1)} &= q_{\omega_1}\cdots q_{\omega_n}&
q_{(2)} &= q_{\omega_{n + 1}}\cdots q_{\omega_{n + m}}\\
r_{(1)} &= \sum_{i = 1}^n p_{\omega_1}\cdots p_{\omega_{i - 1}}r_{\omega_i}q_{\omega_{i + 1}}\cdots q_{\omega_n}&
r_{(2)} &= \sum_{i = 1}^m p_{\omega_{n + 1}}\cdots p_{\omega_{n + i - 1}}r_{\omega_{n + i}}q_{\omega_{n + i + 1}}\cdots q_{\omega_{n + m}}
\end{align*}
so that
\[
u_{(i)}(t) = \frac{p_{(i)}}{q_{(i)}}t + \frac{r_{(i)}}{q_{(i)}}.
\]

The unique fixed point $F_2$ of the contraction $u_{(2)}$ is given by the equation
\[
F_2 = \frac{p_{(2)}}{q_{(2)}}F_2 + \frac{r_{(2)}}{q_{(2)}}
\]
and after solving for $F_2$
\[
F_2 = \frac{r_{(2)}}{q_{(2)} - p_{(2)}}.
\]
In particular, $F_2\in\Q$. Let
\begin{equation}
\label{canonicalform}
p/q = u_{(1)}(F_2) = \frac{p_{(1)}}{q_{(1)}}\frac{r_{(2)}}{q_{(2)} - p_{(2)}} + \frac{r_{(1)}}{q_{(1)}}.
\end{equation}
Here, we mean that $p\in\Z$, $q\in\N$ are the result of adding the fractions in the usual way, without reducing. In particular, $q = q_{(1)}(q_{(2)} - p_{(2)}) \leq q_{(1)}q_{(2)}$. Note that $p/q = \pi(\omega_1^n\concat[\omega_{n + 1}^{n + m}]^\infty) \in J$.

We next want to bound the distance between $x$ and $p/q$. For convenience of notation let $\lambda_{(i)} = |p_{(i)}|/q_{(i)}$ be the contraction ratio of $u_{(i)}$. Now
\[
|y - F_2| = \lambda_{(2)}|z - F_2| \leq \lambda_{(2)}|y - F_2| + \lambda_{(2)}|z - y|.
\]
Solving for $|y - F_2|$ gives
\[
|y - F_2| \leq \frac{\lambda_{(2)}}{1 - \lambda_{(2)}}|z - y|.
\]
Applying $u_{(1)}$ gives
\[
\left|x - p/q\right| \leq \frac{\lambda_{(1)}\lambda_{(2)}}{1 - \lambda_{(2)}}|z - y|
\leq \frac{\lambda_{(1)}\lambda_{(2)}}{1 - \lambda_{(2)}}r_N.
\]
Now $\lambda_{(2)} \leq \max_a \lambda_a < 1$. If we let $K_2 = 1/(1 - \max_a \lambda_a)$ then
\[
\left|x - p/q\right| \leq K_2\lambda_{(1)}\lambda_{(2)}r_N
\]
and on the other hand
\[
q \leq q_{(1)}q_{(2)}.
\]
Expanding and taking logarithms
\begin{align*}
\log|x - p/q| &\leq \log(K_2) + \log(r_N) + \sum_{k = 1}^{n + m}\log(\lambda_{\omega_k})\\
\log(q) &\leq \sum_{k = 1}^{n + m}\log(q_{\omega_k}).
\end{align*}
Now for every $i = 1,\ldots,m$ we have
\[
\log(\lambda_i) \leq (\gamma - 1) \log(q_i)
\]
where $\gamma$ is as in (\ref{gamma}). Thus
\begin{align*}
\log|x - p/q| &\leq \log(K_2) + \log(r_N) + (\gamma - 1) \sum_{k = 1}^{n + m}\log(q_{\omega_k})\\
&\leq \log(K_2) + \log(r_N) + (\gamma - 1) \log(q)
\end{align*}
and exponentiating gives
\[
\left|x - p/q\right| \leq K_2 r_N q^{\gamma - 1}.
\]
Finally, recalling that $\qmax := \max_a q_a$, we have
\[
q \leq \qmax^{n + m}\leq \qmax^N
\]
and so letting $N = \lfloor\log_{\qmax}(Q)\rfloor$ gives $q\leq Q$. Now since $Q\geq\qmax$, we have
\[
r_N \leq (\lfloor\log_{\qmax}(Q)\rfloor/K_1)^{-1/\delta}\leq K_3 \log(Q)^{-1/\delta}
\]
for some $K_3$ sufficiently large. Letting $K = K_2 K_3$ completes the proof.
\end{proof}

\section{Random translations and dilatations} \label{sectiontranslations}
\label{translations}
Fix $d\in\N$ and $E\subseteq \{0,\ldots,d - 1\}$ satisfying $1 < \#(E) < d$, and let
\[
J = \{x\in [0,1]:\textup{ the digits of the base $d$ expansion of $x$ are in $E$}\}.
\]
Such a set $J$ is called a $\times d$-invariant set.

In this section, we consider the image of $J$ under a random affine transformation $f$. We ask whether this set intersects densely ($\cl{f(J)\cap\Q} = f(J)$) with the rational numbers, and if so what the Diophantine properties of the set are. For each $x,y\in\R$, let $f_{x,y}$ be the unique affine map such that $f_{x,y}(x) = 0$ and $f_{x,y}(y) = 1$. Observe that
\[
f_{x,y}^{-1}(\alpha) = (1 - \alpha) x + \alpha y.
\]
\begin{observation}
\label{observationrandomsimilarity}
~
\begin{itemize}
\item[(i)] For $\lambda\times\lambda$-almost every $(x,y)\in\R^2$, we have $f_{x,y}(J)\cap\Q = \emptyset$.
\item[(ii)] For $\mu\times\lambda$-almost every $(x,y)\in\R^2$, we have $f_{x,y}(J)\cap\Q = \{0\}$.
\end{itemize}
Here $\lambda$ is Lebesgue measure on $\R$. In particular, in each of these cases $f_{x,y}(J)$ does not intersect the rational numbers densely.
\end{observation}
\begin{proof}
Since $\Q$ is countable, it is enough to fix $\alpha\in\Q$ and to show that
\begin{itemize}
\item[(i)] For $\lambda\times\lambda$-almost every $(x,y)\in\R^2$, we have $f_{x,y}^{-1}(\alpha)\notin J$.
\item[(ii)] For $\mu\times\lambda$-almost every $(x,y)\in\R^2$, we have $f_{x,y}^{-1}(\alpha)\notin J$ unless $\alpha = 0$.
\end{itemize}
Now (i) follows from the fact that $f_{x,y}^{-1}(\alpha)$ is $\lambda$-random, and $\lambda(J) = 0$. (ii) follows by Fubini's theorem since for $x$ fixed, $f_{x,y}^{-1}(\alpha)$ is $\lambda$-random unless $\alpha = 0$.
\end{proof}

Philosophically, (i) says that the image of $J$ under a random similarity does not intersect $\Q$, and (ii) says that if the image of $J$ under a similarity does intersect $\Q$, then the probability that it does so a second time is zero. Next, we consider the question of whether an image of $J$ which intersects the rationals twice is likely to intersect them a third time. We will assume that the pair of points where $J$ intersects the preimage of $\Q$ is a $\mu\times\mu$-random pair $(x,y)$. Then the similarity must be of the form $T\circ f_{x,y}$, where $T$ preserves the rationals. In particular, $T\circ f_{x,y}(J)$ intersects the rationals a third time if and only if $f_{x,y}(J)$ does, so we will consider only the set $f_{x,y}(J)$ for simplicity.

\begin{theorem}
\label{theoremrandomsimilarity}
Let $J$ be a $\times d$-invariant set. Then for $\mu\times\mu$-almost every $x,y\in J$,
\[
f_{x,y}(J)\cap\Q = \{0,1\}.
\]
In particular, $f_{x,y}(J)$ does not intersect densely with $\Q$.
\end{theorem}
\begin{proof}
It is enough so show that for a fixed $\alpha\in\R\butnot\{0,1\}$, for $\mu\times\mu$-almost every $x,y\in J$ we have $\alpha\notin f_{x,y}(J)$, or equivalently
\begin{equation}
\label{notinJ}
f_{x,y}^{-1}(\alpha) = (1 - \alpha) x + \alpha y \notin J.
\end{equation}
\begin{observation}
There exist $x_0,y_0\in J$ so that $\alpha x_0 + (1 - \alpha) y_0 \in [0,1]\butnot J$.
\end{observation}
\begin{proof}
If $0 < \alpha < 1$, let $(x_0,y_0)$ be a maximal interval contained in the complement of $J$. Then $x_0,y_0\in J$, but $\alpha x_0 + (1 - \alpha) y_0\in [0,1]\butnot J$.

If $\alpha < 0$, then let $(w_0,x_0)$ be a maximal interval contained in the complement of $J$. Then there exist $y_0\in J$ with $y_0 > x_0$ and $|y_0 - x_0|$ arbitrarily small. For such $y_0$, we have $\alpha x_0 + (1 - \alpha) y_0 < x_0$, and if $|y_0 - x_0|$ is small enough then $\alpha x_0 + (1 - \alpha) y_0 > w_0$, which implies $\alpha x_0 + (1 - \alpha) y_0\in [0,1]\butnot J$.

The case $\alpha > 1$ is similar to the case $\alpha < 0$.
\end{proof}
Let $\omega^{(0)},\tau^{(0)}\in E^\N$ be such that $x_0 = \sum_{n\in\N}\omega_n^{(0)} d^{-n}$ and $y_0 = \sum_{n\in\N}\tau_n^{(0)} d^{-n}$. Then
\[
z_0 := \alpha x_0 + (1 - \alpha) y_0 = \sum_{n\in\N} (\alpha \omega_n^{(0)} + (1 - \alpha) \tau_n^{(0)}) d^{-n} \in [0,1]\butnot J.
\]
Let $N$ be large enough so that $d^{-N} < \dist(z_0,J)$. Then
\[
[z_0',z_0' + d^{-N}]\cap J = \emptyset \textup{ where } z_0' = \sum_{n\leq N} (\alpha \omega_n^{(0)} + (1 - \alpha) \tau_n^{(0)}) d^{-n}.
\]
In particular, the iterated function system generated by the maps $(u_\omega)_{\omega\in E^N}$ together with the map
\[
u_0(t) = z_0' + d^{-N} t
\]
satisfies the open set condition (the open set is any sufficiently small neighborhood of $[0,1]$). Let $M$ be large enough so that
\begin{equation}
\label{morethan2times}
(\#(E^N) + 1)^M > 4\#(E^N)^M.
\end{equation}
Let $\omega,\tau\in E^\N$ be $\mu$-random, and let $x = \sum_{n\in\N}\omega_n d^{-n}$, $y = \sum_{n\in\N}\tau_n d^{-n}$. Then
\[
(1 - \alpha) x + \alpha y = \sum_{n\in\N} (\alpha \omega_n + (1 - \alpha) \tau_n) d^{-n}.
\]
For each $n\in\N$ let
\[
z_n = \sum_{i\leq MNn} (\alpha \omega_i + (1 - \alpha) \tau_i) d^{-i}.
\]
\begin{claim}
For each $n\in\N$,
\[
\prob\left([z_n,z_n + d^{-MNn}]\cap J = \emptyset \given (\omega_i,\tau_i)_{i = 1}^{MN(n - 1)} \right) \geq \#(E)^{-2MN}.
\]
\end{claim}
\begin{subproof}
Fix $(\omega_i,\tau_i)_{i = 1}^{MN(n - 1)}$. To prove the claim, we need to show that there exists $(\omega_i,\tau_i)_{i = MN(n - 1) + 1}^{MNn}$ such that
\begin{equation}
\label{interval}
\left[z_{n - 1} + \sum_{i = MN(n - 1) + 1}^{MNn} (\alpha \omega_i + (1 - \alpha) \tau_i) d^{-i},\cdots + d^{-MNn}\right]
\end{equation}
is disjoint from $J$. (The lower bound on the probability comes from the fact that there are only $\#(E)^{2MN}$ possible sequences $(\omega_i,\tau_i)_{i = MN(n - 1) + 1}^{MNn}$.) 

Consider the set of all sequences $(\omega_i,\tau_i)_{i = MN(n - 1) + 1}^{MNn}$ such that for each $j = 1,\ldots,M$, either $\omega_i = \tau_i$ for all $i = MN(n - 1) + N(j - 1) + 1,\ldots, MN(n - 1) + Nj$, or
\[
(\omega_{MN(n - 1) + N(j - 1) + i},\tau_{MN(n - 1) + N(j - 1) + i}) = (\omega_i^{(0)},\tau_i^{(0)}) \text{ for } i = 1,\ldots,N.
\]
The number of such sequences is clearly $(\#(E^N) + 1)^M$. On the other hand, since the IFS $(u_\omega)_{\omega\in E^\N}\cup(u_0)$ satisfies the open set condition, the intervals \eqref{interval} are disjoint for $(\omega_i,\tau_i)_{i = MN(n - 1) + 1}^{MNn}$ in this collection. Thus, each $MNn$-level interval for $J$ can intersect at most two of these intervals; moreover, only intervals which intersect the interval $[z_{n - 1},z_{n - 1} + d^{-MN(n - 1)}]$ can intersect any of these intervals. There are at most $2\#(E^{MN})$ such intervals, so by \eqref{morethan2times}, one of these intervals is disjoint from every $MNn$-level interval of $J$, which implies that it is disjoint from $J$.
\end{subproof}
By the Borel-Cantelli lemma, for $\mu\times\mu$-almost every pair $(x,y)$, we have
\[
[z_n,z_n + d^{-MNn}]\cap J = \emptyset
\]
for infinitely many $n\in\N$. In particular, $z = (1 - \alpha) x + \alpha y$ is contained in this interval, so $z\notin J$.
\end{proof}

Although the image of $J$ under a random affine transformation does not have dense intersection with $\Q$, the story is different if we consider only translations. For each $x\in\R$, let $f_x$ be the unique translation so that $f_x(x) = 0$, i.e. $f_x(t) = t - x$, and let $J_x = f_x(J)$.

\begin{observation}
\label{observationrandomtranslation}
For $\lambda$-almost every $x\in\R$, $J_x\cap\Q = \emptyset$.
\end{observation}
\begin{proof}
It suffices to show that for each $\alpha\in\R$, the probability that $f_x^{-1}(\alpha)\in J$ is zero; this is obvious since $\lambda(J) = 0$.
\end{proof}

However, unlike the situation with a random similarity, if a translation of $J$ intersects $\Q$ in at least one point, then it intersects $\Q$ densely:

\begin{lemma}
\label{dense}
For all $x\in\R$, if $\Q\cap J_x$ is nonempty, then $\Q\cap J_x$ is dense in $J_x$. Furthermore, there exists $K < \infty$ such that every point $y\in J_x$ is intrinsically approximable with respect to the constant function
\begin{equation}
\label{Kq}
\psi(q) = K.
\end{equation}
\end{lemma}

\begin{proof}
Fix $y\in J_x$ and $p_0/q_0\in\Q\cap J_x$. For each $n\in\N$, let $z_n$ be the number whose base $d$ expansion is found by concatenating the first $n$ digits of $x + y \in J$ with the remaining digits of $x + p_0/q_0\in J$, starting from the $(n + 1)$st. We observe that $z_n\in J$, since the digits of the base $d$ expansion of $z_n$ lie in $E$. Then $r_n := z_n - x - p_0/q_0$ is a multiple of $1/d^n$. In particular, $p_0/q_0 + r_n = z_n - x \in \Q\cap J$ and
\[
\denom(p_0/q_0 + r_n)\leq q_0 d^n.
\]
Now
\[
|y - (p_0/q_0 + r_n)| = |z_n - (x + y)| \leq d^{-n}
\]
since $z$ agrees with $(x + y)$ up to the first $n$ digits. Thus
\[
|y - (p_0/q_0 + r_n)| \leq \frac{q_0}{\denom(p_0/q_0 + r_n)},
\]
proving the lemma with $K = q_0$.
\end{proof}

For the remainder of this section, we will consider a set of the form $J_x$ where $x\in J$. Such a set is representative of a translate of $J$ intersecting $\Q$, because any intersection of $\Q$ with $J$ can be translated to $0$.

We now discuss optimality of the approximation function (\ref{Kq}). We consider the following case:
\begin{enumerate}[i)]
\item $d$ is prime
\item The set $E$ of allowable digits contains no two adjacent integers, and contains both $0$ and $(d - 1)$.\comdavid{Added, since we are using it in the proof.}
\end{enumerate}
We remark that (ii) implies that the coding map $\pi:E^\N\to J$ is injective.

Let $J$ satisfy (i)-(ii), and let $x\in J$ be a $\mu$-random point. Then the approximation function (\ref{Kq}) is optimal in the following two senses:
\begin{itemize}
\item The set of badly intrinsically approximable numbers is of full Hausdorff dimension (Theorem \ref{theoremcase0ba})
\item A Khinchin-type theorem holds (Theorem \ref{theoremcase0khinchin})
\end{itemize}

To prove these results, we begin with the following theorem:
\begin{theorem}
\label{exceptions}
Suppose that $J$ is a $\times d$-invariant limit set satisfying \textup{(i)-(ii)}. If $x\in J$ is a $\mu$-random point, then
\begin{itemize}
\item[iii)] Every rational in $J_x$ is of the form $p/d^n$ for some $p,n\in\N$.
\end{itemize}
\end{theorem}
\begin{proof}
Suppose that $x\in J$ is a $\mu$-random point, and let $x = \sum_{i = 1}^\infty a_i d^{-i}$ be the base $d$ expansion of $x$. Then the sequence $(a_i)_i$ is an independent and identically distributed sequence of random variables whose common distribution is the uniform distribution on $E$. By the Law of Large Numbers, every finite word in $E^*$ is a subword of the infinite word $(a_i)_i$.

By way of contradiction, suppose that there exists a rational $p/q\in\Q\cap J_x$ whose denominator is not a power of $d$. Without loss of generality assume $p/q > 0$; the case $p/q < 0$ is similar. Let $p/q = \sum_{i = 1}^\infty b_i d^{-i}$ be the base $d$ expansion of $p/q$, and let $x + p/q = \sum_{i = 1}^\infty c_i d^{-i} \in J$. Then $c_i = a_i + b_i$ or $c_i = a_i + b_i + 1$, mod $d$, depending on whether there is a carry from the lower level terms. Write
\[
p/q = .b_1\ldots b_{n + m}b_{n + 1}\ldots \hspace{1 in}\text{(in base $d$)}
\]
for some $n,m\in\N$. Let $k = b_{n + m}$. Fix $a\in E\butnot\{d - 1 - k\}$. Then by condition (ii), either $a + k \notin E$, or $a + k + 1 \notin E$ (mod $d$). Without loss of generality suppose that $a + k \notin E$. Let $\tau = (a 0^m)^{n + m}$. (If $a + k + 1\notin E$ we take $\tau = (a (d - 1)^m)^{n + m}$.) By assumption $\tau$ is a substring of $(a_i)_i$. Thus we can find $a 0^m$ as a substring of $(a_i)_i$ such that the $a$ corresponds to an occurence of $k$ in $p/q$, i.e. there exists $i\in\N$ such that
\begin{align*}
a_{n + mi} &= a\\
a_{n + mi + j} &= 0
\end{align*}
for all $j = 1,\ldots,m$. Now since $q$ is not a power of $d$, we have $b_{n + j}\neq d - 1$ for some $j = 1,\ldots,m$. Thus the zeros $(a_{n + mi + j})_{j = 1}^m$ are sufficient to ensure that there is not a carry in the $(n + mi)$th place, implying $c_{n + mi} = a_{n + mi} + b_{n + mi} = a + k \notin E$, a contradiction. Thus $p/q$ does not exist.
\end{proof}

\begin{lemma}
\label{lemmacase0ba}
Suppose that $J$ satisfies \textup{(i)-(ii)}, and let $x = \pi(\omega)\in J$ be a $\mu$-random point. Then for every $\psi:\N\rightarrow(0,\infty)$, a point $y = \pi(\tau) - x \in J_x$ is badly intrinsically approximable with respect to $\psi$ if and only if there exists $K < \infty$ such that for every pair $(n,r)\in\N^2$ satisfying $\omega_{n + 1}^{n + r} = \tau_{n + 1}^{n + r}$, we have
\[
r\leq K + \Psi(n).
\]
Here we use the notation
\begin{equation}
\label{Psidefcase0}
\Psi(n) := -\log_d(\psi(d^n)).
\end{equation}
\end{lemma}
\begin{proof}
Suppose that $y$ is intrinsically well approximable with respect to $\psi$. Then for all $\varepsilon > 0$ there exist infinitely many rational approximations $p/q = \pi(\eta) - x \in\Q\cap J_x$ satisfying $|y - p/q| \leq \varepsilon\psi(q)/q$. By condition (iii), $p/q$ can be written in the form $p/d^n$ for some $p,n\in\N$. On the other hand, since $d$ is prime, it follows that the reduced form of the fraction $p/d^n$ is also of the form $p/d^n$ (possibly with different $p,n$). In other words, $\qred = d^n$ for some $n\in\N$. Now since $\pi(\eta) - \pi(\omega) = p/q$, we have that $\eta$ agrees with $\omega$ except for the first $n$ digits. On the other hand, we have $|\pi(\tau) - \pi(\eta)| \leq \varepsilon\psi(d^n)/d^n$, which implies that $\tau$ and $\eta$ agree on the first $\lfloor\log_{1/d}(\varepsilon\psi(d^n)) + n\rfloor - 1$ digits (here we are using condition (ii)). Thus $\omega_{n + 1}^{n + r} = \tau_{n + 1}^{n + r}$, where $r = \lfloor\log_{1/d}(\varepsilon\psi(d^n))\rfloor - 1$. Thus for all $K < \infty$, there exist infinitely many pairs $(n,r)\in\N^2$ such that $\omega_{n + 1}^{n + r} = \tau_{n + 1}^{n + r}$ but $r\geq K + \log_{1/d}(\psi(d^n))$.

On the other hand, suppose that for all $K < \infty$, there exist infinitely many such pairs. For each pair $(n,r)$, if we define $\eta$ to be the string which agrees with $\tau$ for the first $n$ digits but then agrees with $\omega$, we find that the rational approximation $p/q := \pi(\eta) - x \in \Q\cap J_x$ satisfies $|y - p/q| \leq \varepsilon\psi(q)/q$.
\end{proof}

As an immediate consequence, we get the optimality of the Dirichlet function $\psi(q) = 1$:

\begin{theorem}[Optimality]
\label{theoremcase0ba}
Suppose that $J$ satisfies \textup{(i)-(ii)}, and let $x\in J$ be a $\mu$-random point. Then the set of numbers $y\in J_x$ which are badly intrinsically approximable with respect to the approximation function $\psi(q) = K$ is of Hausdorff dimension $\delta$.
\end{theorem}
\begin{proof}
We have $\Psi(n) = 0$, so $y = \pi(\tau) - x\in J_x$ is badly approximable with respect to (\ref{Kq}) if and only if the length of the strings on which $\omega$ and $\tau$ agree is uniformly bounded. Thus for every $k\in\N$, the set
\[
S_k := \{\tau\in E^\N: \tau_{kn}\neq\omega_{kn}\all n\in\N\}
\]
is contained in the set of badly approximable points. On the other hand, it is readily computed (using e.g. Hutchinson's formula \cite{Hutchinson}) that the Hausdorff dimension of $S_k$ tends to $\delta$ as $k$ tends to infinity.
\end{proof}

Finally, using a theorem of Levesley, Salp, and Velani \cite{LSV}, we are able to prove the following theorem, which incorporates both a Khinchin-type and a Jarnik-Besicovitch-type theorem:

\begin{theorem}
\label{theoremcase0khinchin}
Suppose that $J$ satisfies \textup{(i)-(ii)}, and let $x\in J$ be a $\mu$-random point. Fix any function $\psi:\N\rightarrow(0,\infty)$. Let $f$ be a dimension function such that $t\mapsto t^{-\delta}f(t)$ is monotonic. If we denote the set of $\psi$-intrinsically well approximable points by $\WA_{\psi,\mathrm{int}}$, then
\[
\HH^f(\WA_{\psi,\mathrm{int}}) = \begin{cases}
0 & \textup{if }\sum_{n = 1}^\infty f(\psi(d^n)/d^n) (d^n)^\delta < \infty\\
\HH^f(J) & \textup{if }\sum_{n = 1}^\infty f(\psi(d^n)/d^n) (d^n)^\delta = \infty
\end{cases},
\]
where $\HH^f$ is Hausdorff $f$-measure.
\end{theorem}

\begin{corollary}
\label{corKhinchin}
If $\psi(q) = \log(q)^{-1/\delta}$, then almost every point is intrinsically well approximable with respect to $\psi$; if $\psi(q) = \log(q)^{-(1 + \varepsilon)/\delta}$, then almost every point is badly intrinsically approximable with respect to $\psi$.
\end{corollary}
~
\begin{proof}[Proof of Theorem \ref{theoremcase0khinchin}]
We will use the following theorem from \cite{LSV}:
\begin{theorem}[{\cite[Theorem 1]{LSV}}]
\label{theoremLSV}
For any approximation function $\psi$, consider the set
\[
\WA_{\psi,\mathrm{term}} := \{x\in [0,1]: |x - p/q| < \psi(q)\textup{ for infinitely many }(p,q)\in\N\times d^\N\},
\]
i.e. the set of all points which are $\psi$-approximable with respect to the rationals with terminating base $d$ expansions.

Now suppose that $J$ is a $\times d$-invariant limit set satisfying \textup{(i)-(ii)}. Let $f$ be a dimension function such that $t\mapsto t^{-\delta}f(t)$ is monotonic. Then
\[
\HH^f(\WA_{\psi,\mathrm{term}}\cap J) = \begin{cases}
0 & \text{if }\sum_{n = 1}^\infty f(\psi(d^n))\times (d^n)^\delta < \infty\\
\HH^f(J) & \text{if }\sum_{n = 1}^\infty f(\psi(d^n))\times (d^n)^\delta = \infty
\end{cases}.
\]
\end{theorem}
In \cite{LSV} the theorem is stated in the case $d = 3$, $J$ the Cantor set, but the proof clearly generalizes.

Let us call a point $y\in J$ \emph{badly terminally approximable with respect to $\psi$} if $y\notin\WA_{\varepsilon\psi,\mathrm{term}}$ for some $\varepsilon > 0$. Clearly, the proof of Lemma \ref{lemmacase0ba} generalizes to the following statement:
\begin{lemma}
Suppose that $J$ is as above. Then for every $\psi:\N\rightarrow(0,\infty)$, a point $y = \pi(\tau)\in J$ is badly terminally approximable with respect to $\psi$ if and only if there exists $K < \infty$ such that for every pair $(n,r)\in\N^2$ such that $\tau_{n + 1}^{n + r}$ is all $0$s or all $1$s, we have
\[
r \leq K + \Psi(n).
\]
\end{lemma}
Now let $x = \pi(\omega)\in J$ satisfy (iii). Define an automorphism $\Phi:E^\N\rightarrow E^\N$ by the following procedure:
\begin{itemize}
\item For each $n\in\N$, choose a permutation $\Phi_n$ of $E$ such that $\Phi_n(\omega_n) = 0$.
\item Let $\Phi(\tau) = (\Phi_n(\tau_n))_n$.
\end{itemize}
Clearly, $\Phi$ is an isometry of $E^\N$, and so the map $\w{\phi}:J_x\rightarrow J$ defined by
\[
\w{\Phi}(y) = \pi(\Phi(x + y))
\]
is bi-Lipschitz. Furthermore, $\w{\Phi}$ sends rational points of $J_x$ to left endpoints of $J$, whose denominator is the same up to a constant.

Observe now that a point $y\in J_x$ is badly intrinsically approximable with respect to an approximation function $\psi$ if and only if both $\w{\Phi}(y)$ is badly terminally approximable with respect to $q\mapsto q\psi(q)$ (the factor comes from a difference in notation between our paper and \cite{LSV}). Thus, the Hausdorff measure of the badly intrinsically approximable points agrees up to a constant with the Hausdorff measure of the badly terminally approximable points. Applying Theorem \ref{theoremLSV} completes the proof.
\end{proof}

\section{The intrinsic denominator} \label{sectionintrinsicdenominator}
In this section we assume that $J$ is a limit set of a rational IFS satisfying the open set condition.

Suppose that $p/q\in\Q\cap J$ is the image of the eventually periodic word $\omega\in E^\N$. Fix $n\in\N$ so that $\sigma^n(\omega)$ is periodic, and let $m$ be a period of $\sigma^n(\omega)$, so that
\[
\omega = \omega_1\ldots\omega_n\omega_{n + 1}\ldots\omega_{n + m}\omega_{n + 1}\ldots.
\]
Based on $\omega$, $n$, and $m$, we can define $p_{(i)}$, $q_{(i)}$, $r_{(i)}$, $i = 1,2$ as in the proof of Theorem \ref{theoremdirichlet} and we define the \emph{intrinsic denominator of $p/q\in J$ with respect to the triple $(\omega,n,m)$} to be the denominator of (\ref{canonicalform}), i.e. the intrinsic denominator is the number
\[
q_{(1)}(q_{(2)} - p_{(2)}).
\]
\begin{observation}
\label{observationintrinsicdenominator}
If $\omega$ is fixed, then the intrinsic denominator is minimized when $n$ and $m$ are minimal. Furthermore, this intrinsic denominator divides the intrinsic denominator of $p/q$ with respect to any other pair $(\w{n},\w{m})$.
\end{observation}
Thus, we define the \emph{intrinsic denominator of $p/q$ with respect to $\omega$} to be the intrinsic denominator with respect to $(\omega,n,m)$, with $n$ and $m$ minimal. The intrinsic denominator of $p/q$ with respect to $\omega$ represents ``everything that the symbolic representation $p/q = \pi(\omega)$ can tell us about the denominator of $p/q$''.

In general, a fixed rational number could have more than one eventually periodic symbolic representation, or it could have none at all. However, we do not know of any examples of rationals with symbolic representations which are not eventually periodic. Furthermore, in most of the cases that we care about, this is impossible:

\begin{lemma}
\label{lemmacanonical}
Suppose that $p_a = \pm 1$ for all $a\in E$, where $p_a$ are given by \textup{(\ref{rationalIFS})}. Then every rational in $J$ is the image of an eventually periodic word (and therefore has an intrinsic denominator).
\end{lemma}

\begin{remark}
In the case where $J$ is $\times d$-invariant, then this lemma is well-known (it asserts that every rational has an eventually periodic $d$-ary representation). However, the lemma does not appear to be well-known e.g. for the IFS ``$1/3,1/4$''
\begin{align*}
u_0(x) &= x/3\\
u_1(x) &= x/4 + 3/4.
\end{align*}
\end{remark}

\begin{proof}[Proof of Lemma \ref{lemmacanonical}]
For each $a\in E$, since $p_a = \pm 1$ we can write
\[
u_a^{-1}(x) = \pm(q_a x - r_a).
\]
In particular, if $x$ is rational then the denominator of $u_a^{-1}(x)$ divides the denominator of $x$. Thus if $\omega\in E^\N$ is a preimage of $x$ under $\pi$, then the forward orbit $(\sigma^n(\omega))_n$ lies in the set $\pi^{-1}\{p/q:\text{$q$ divides the denominator of $x$}\}$, which is finite by the open set condition. Thus $\omega$ is eventually periodic.
\end{proof}

For the remainder of this section, we will restrict ourselves to the case where $p_a = \pm 1$ for all $a\in E$, so that every rational in $J$ has at least one intrinsic denominator. We will also impose the following condition which guarantees that no rational can have more than one intrinsic denominator:
\begin{definition}
The IFS $(u_a)_{a\in E}$ satisfies the \emph{strong separation condition} if there exists a closed interval $[c,d]$ such that the collection $(u_a([c,d]))_{a\in E}$ is a disjoint collection of subsets of $[c,d]$.
\end{definition}
For example, the IFS for the Cantor set satisfies the strong separation condition. Note that the strong separation condition implies the open set condition, since we can take our open set to be $(c,d)$.

From now on, we will assume that our IFS satisfies the strong separation condition. Since this condition implies that every element of $J$ has exactly one symbolic representation, it follows that every rational in $J$ has exactly one intrinsic denominator.
\begin{notation}
The intrinsic denominator of $p/q\in J$ will be denoted $\qint$, whereas the denominator of $p/q$ in reduced form will be denoted $\qred$.
\end{notation}
\begin{observation}
We have $\qred \divides \qint$.
\end{observation}
\begin{definition}
\label{badsymbolic}
Let $\psi:(0,\infty)\rightarrow(0,\infty)$ be a nonincreasing function. A point $x\in J$ is said to be \emph{badly symbolically approximable with respect to $\psi$} if there exists $\varepsilon > 0$ such that for all $p/q\in\Q\cap J$ we have
\[
\left|x - p/q\right| \geq \varepsilon\frac{\psi(\qint)}{\qint}.
\]
Otherwise, $x$ is said to be \emph{symbolically well approximable with respect to $\psi$}.
\end{definition}
So, badly intrinsically approximable implies badly symbolically approximable, but not vice-versa.

\begin{notation}
If $\omega\in E^r$ is a finite word, we define the \emph{pseudolength} of $\omega$ to be the number
\[
\pseud(\omega) := \sum_{i = 1}^r \log(q_{\omega_i}).
\]
In the case where the set $J$ is $\times d$-invariant for some $d$, the pseudolength of $\omega$ is just equal to $\log(d)$ times the length of $\omega$.
\end{notation}

\begin{lemma}
\label{lemmaba}
Suppose that
\begin{enumerate}[i)]
\item $\psi$ is \emph{slowly varying} i.e. $\psi(Kq)\asymp \psi(q)$ for all $K > 0$
\item $\psi$ is bounded
\end{enumerate}
Then for all $x = \pi(\omega) \in J$, $x$ is badly symbolically approximable with respect to $\psi$ if and only if there exists $K < \infty$ such that for every finite word $\eta$ of length $r$ which occurs twice (possibly overlapping) in the initial segment $\omega_1^\ell$,\footnote{In the sequel we shall call such an $\eta$ simply a ``repeat $\eta$''.} we have
\begin{equation}
\label{badlyapproximable4}
\pseud(\eta) \leq K + \Psi(\pseud(\omega_1^{\ell - r})),
\end{equation}
where
\[
\Psi(t) := -\log(\psi(e^t)).
\]
\end{lemma}
Note that if $\psi(q) = \log(q)^{-s/\delta}$ then $\Psi(t) = s\log(t)/\delta$.
\begin{proof}
Suppose that $x\in J$ is symbolically well approximable with respect to $\psi$. Then for all $\varepsilon > 0$ there exist infinitely many rational approximations $p/q\in \Q\cap J$ satisfying $|x - p/q| \leq \varepsilon \psi(\qint)/\qint$. Fix such a $p/q$, and let $\tau\in E^\N$ be its preimage. Let $n,m\in\N$ be minimal such that
\[
\tau = \tau_1\ldots \tau_{n + m}\tau_{n + 1}\ldots
\]
According to (\ref{canonicalform}), the intrinsic denominator of $p/q$ is equal to
\[
\qint := \left(\prod_{i = 1}^n q_{\tau_i}\right)\left(\prod_{i = 1}^m q_{\tau_{n + i}} \pm 1\right) \asymp \prod_{i = 1}^{n + m}q_{\tau_i}.
\]
On the other hand, if $\ell$ is the largest integer for which $\omega_\ell = \tau_\ell$, then
\[
\left|x - p/q\right| \asymp \prod_{i = 1}^\ell\frac{1}{q_{\tau_i}}.
\]
Here we have used the strong separation condition to get the lower bound. Thus we have
\[
\prod_{i = 1}^\ell\frac{1}{q_{\tau_i}} \lesssim \varepsilon\left(\prod_{i = 1}^{n + m}\frac{1}{q_{\tau_i}}\right)\psi\left(\prod_{i = 1}^{n + m}q_{\tau_i}\right).
\]
Here we have used the slowly varying condition (i).

Since $\psi$ is bounded, by choosing $\varepsilon$ small enough we can force $\prod_{i = 1}^\ell(1/q_{\tau_i}) < \prod_{i = 1}^{n + m}(1/q_{\tau_i})$, which implies $\ell > n + m$. Thus we have
\begin{equation}
\label{bintermsofa}
\tau = \omega_1\ldots\omega_{n+m}\omega_{n+1}\ldots\omega_{n+m}\omega_{n+1}\ldots
\end{equation}
and in particular
\[
\prod_{i = n + m + 1}^\ell\frac{1}{q_{\omega_i}} \lesssim \varepsilon\psi\left(\prod_{i = 1}^{n + m}q_{\omega_i}\right).
\]
and taking negative logarithms yields
\[
\pseud(\omega_{n + m + 1}^\ell) \gtrsim_+ -\log(\varepsilon) + \Psi(\pseud(\omega_1^{n + m})).
\]
Since $\omega_{n + 1}^{\ell - m} = \tau_{n + 1}^{\ell - m} = \tau_{n + m + 1}^\ell = \omega_{n + m + 1}^\ell$, it follows that there are infinitely many words $\eta$ which are repeated in $\omega$ but do not satisfy (\ref{badlyapproximable4}).

On the other hand, suppose that for all $K < \infty$ there exist infinitely many words $\eta$ which are repeated in $\omega$ but do not satisfy (\ref{badlyapproximable4}). For each such $\eta$, let $r$ be the length of the $\eta$, and let $0 \leq n < n + m$ be the places where it occurs, so that $\eta = \omega_{n + 1}^{n + r} = \omega_{n + m + 1}^{n + m + r}$. Let $\tau$ be defined by (\ref{bintermsofa}), and let $p/q = \pi(\tau)$. Then $\omega$ and $\tau$ agree up to at least $\ell := n + m + r$ places, and a reverse calculation yields that $|x - p/q| \lesssim e^{-K}\psi(\qint)/\qint$. Thus $x$ is symbolically well approximable with respect to $\psi$.
\end{proof}

We now discuss the approximability of a $\mu$-random number $x\in J$. In the following discussion $x$ will always denote a $\mu$-random number, and $\omega$ will denote its preimage under $\pi$. We write $\prob$ for probability and $\scrE$ for expected value, so that $\prob(x\in S) = \mu(S)$ and $\scrE[f(x)] = \int f(t)\dee\mu(t)$. In particular, the sequence $(\omega_i)_i$ is a sequence of independent and identically distributed random variables, whose distribution is given by
\[
\prob(\omega_n = a) = q_a^{-\delta} = e^{-\delta\pseud(a)}.
\]

\begin{lemma}
\label{lemmaprobability}
Fix $n,m\in\N$ and $\ell_0 > 0$. Let
\[
E_{n,m,\ell_0} := \{\eta\in E^*:\text{there exists $r$ such that $\eta_{n + 1}^{n + r} = \eta_{n + m + 1}^{n + m + r}$ and such that $\pseud(\eta_{n + 1}^{n + r})\geq \ell_0$}\}.
\]
Then $\prob(\omega\in E_{n,m,\ell_0}) \leq e^{-\delta\ell_0}$.
\end{lemma}
\begin{remark}
It is possible that $r > m$, so that $\omega_{n + 1}^{n + r}$ and $\omega_{n + m + 1}^{n + m + r}$ overlap. Thus a naive independence argument does not work.
\end{remark}
\begin{proof}[Proof of Lemma \ref{lemmaprobability}]
Without loss of generality suppose $n = 0$. For each $\eta\in E^{m + N}$ let
\[
\phi(\eta) := \begin{cases}
e^{\delta\ell_0} & \eta\in E_{0,m,\ell_0}\\
0 & \eta_1^N \neq \eta_{m + 1}^{m + N}\text{ and }\eta\notin E_{0,m,\ell_0}\\
e^{\delta\pseud(\eta_{m + 1}^{m + N})} & \text{otherwise}
\end{cases}.
\]Then for any $\eta\in E^m$, we have $\phi(\eta) = 1$, and for any $\eta\in E^{m + N}$, we have
\[
\scrE[\phi(\omega_1^{m + N + 1})\given \omega_1^{m + N} = \eta] \leq \phi(\eta).
\]
A simple induction therefore yields $\scrE[\phi(\omega_1^{m + N})] \leq 1$. The result therefore follows from Markov's inequality.
\end{proof}

\begin{theorem}[A Khinchin-type theorem for fractals]
\label{theoremkhinchin}
Suppose that $(u_a)_{a\in E}$ and $\psi$ are such that the hypotheses \textup{(i) - (ii)} of Lemma \ref{lemmaba} are satisfied. Also suppose that $\psi$ is nonincreasing.
\begin{enumerate}[i)]
\item If the series
\begin{equation}
\label{qdeltalnq}
\sum_{q = 1}^\infty \frac{\log(q)\psi(q)^\delta}{q}
\end{equation}
converges, then for $\mu$-almost every $x\in J$, $x$ is badly symbolically approximable with respect to $\psi$.
\item If the series
\begin{equation}
\label{qdeltalnqlnlnq}
\sum_{q = 1}^\infty \frac{\log(q)\psi(q)^\delta}{q\log(\psi(q))}
\end{equation}
diverges, then for $\mu$-almost every $x\in J$, $x$ is symbolically well approximable with respect to $\psi$.
\end{enumerate}
\end{theorem}
In particular, if $\psi(q) = \log(q)^{-(2/\delta + \varepsilon)}$, then case (i) holds, and if $\psi(q) = \log(q)^{-2/\delta}$, then case (ii) holds.

\begin{corollary}
\label{KhinchinCantor}
Let $C$ be the Cantor set and $\mu$ the Hausdorff measure in the Cantor's set dimension restricted to $C$. Then for $\mu$-almost every $x\in J$, $x$ is badly symbolically approximable with respect to $\psi(q) = \log(q)^{-(2/\delta + \varepsilon)}$ and $\mu$-almost every $x\in J$, $x$ is symbolically well approximable with respect to $\psi(q) = \log(q)^{-2/\delta}.$
\end{corollary}
~
\begin{proof}[Proof of Theorem \ref{theoremkhinchin}]~
\begin{enumerate}[i)]
\item
Fix $K$ to be determined. For each $n,m\in\N$ let $\ell_{n,m} = K + \Psi(n + m)$. By Lemma \ref{lemmaprobability} we have
\begin{equation}
\label{easybound}
\prob\left(\bigcup_{n,m\in\N}E_{n,m,\ell_{n,m}}\right)\leq \sum_{n,m\in\N}e^{-\delta(K + \Psi(n + m))} = e^{-\delta K}\sum_{n\geq 2}(n - 1) e^{-\delta\Psi(n)}.
\end{equation}
If (\ref{qdeltalnq}) converges, then the series
\[
\sum_{n = 1}^\infty n e^{-\delta\Psi(n)}
\]
also converges. Thus for all $\varepsilon > 0$ there exists $K < \infty$ such that the right hand side of (\ref{easybound}) is at most $\varepsilon$. In particular, the probability that $\omega\in\bigcup_{n,m\in\N}E_{n,m,\ell_{n,m}}$ can be made arbitrarily small. By Lemma \ref{lemmaba}, this implies that if $x$ is $\mu$-random, then $x$ is badly intrinsically approximable with respect to $\psi$.
\item
Let $\alpha$ and $\beta$ be the maximum and minimum pseudolengths of a single letter, respectively.

Fix $K < \infty$. Choose a random $\omega\in E^\N$. Fix $t\in\N$. For each $N\in\N$, we denote by $s(N)$ the smallest integer such that
\[
\pseud(\omega_N^{s(N) - 1}) \geq \ell_t := K + \Psi(\alpha 2^{2t + 2}).
\]
We note that for each $N\in\N$, the string $\omega_N^{s(N) - 1}$ lies in the set
\[
E^{\ell_t} := \{\eta\in E^r: \pseud(\eta)\geq \ell_t\text{ but }\pseud(\eta_1^{r - 1}) < \ell_t\}.
\]
Consider the event
\begin{quote}
\begin{itemize}
\item[$E_t$:] For all $N_1,N_2$ distinct with $2^{2t} \leq N_i < s(N_i) \leq 2^{2t + 2}$ we have $\omega_{N_1}^{s(N_1) - 1} \neq \omega_{N_2}^{s(N_2) - 1}$.
\end{itemize}
\end{quote}
We note that if (\ref{badlyapproximable4}) holds, then $E_t$ must hold for all $t\in\N$, due to our choice of $\ell_t$. Furthermore, the event $E_t$ depends only on the string $\omega_{2^{2t}}^{2^{2t + 2} - 1}$, and therefore the events $(E_t)_t$ are independent. In what follows, we will prove an upper bound on $\prob(E_t)$.

We begin by dividing $\omega_{2^{2t}}^{2^{2t + 1} - 1}$ into a sequence of subwords $(\omega_{N_{t,i}}^{N_{t,i + 1} - 1})_i$ in the following manner: Let $N_{t,0} = 2^{2t}$, and if $N_{t,i}$ has been chosen, then let $N_{t,i + 1} = s(N_{t,i})$. The sequence $(\omega_{N_{t,i}}^{N_{t,i + 1} - 1})_i$ is independent and identically distributed with distribution $\prob(\omega_{N_{t,i}}^{N_{t,i + 1} - 1} = \eta) = e^{-\delta\pseud(\eta)}$.

Now for all $\eta\in E^{\ell_t}$, we have $\pseud(\eta)\leq \ell_t + \alpha$. Thus $\pseud(\omega_{2^{2t}}^{N_{t,i} - 1})\leq i(\ell_t + \alpha)$ for all $i$. Let
\[
N_t = \left\lfloor \frac{2^{2t}\beta}{\ell_t + \alpha}\right\rfloor.
\]
Then $\pseud(\omega_{2^{2t}}^{N_{t,N_t} - 1})\leq 2^{2t}\beta$, and so $N_{t,N_t} - 2^{2t}\leq 2^{2t}$ i.e. $N_{t,N_t} \leq 2^{2t + 1}$. It follows that the sequence $(\omega_{N_{t,i}}^{N_{t,i + 1} - 1})_{i = 0}^{N_t - 1}$ depends only on the string $\omega_{2^{2t}}^{2^{2t + 1} - 1}$.

Fix a string $\tau$ of length $2^{2t + 1}$. We will prove an upper bound on $E_t$ conditioned on the event $\omega_{2^{2t + 1}}^{2^{2t + 2} - 1} = \tau$, which will then yield the unconditional bound we desire.

If $\tau$ contains two identical substrings which are members of $E^{\ell_t}$, then the event $\omega_{2^{2t + 1}}^{2^{2t + 2} - 1} = \tau$ contradicts $E_t$, so that $\prob(E_t\given \omega_{2^{2t + 1}}^{2^{2t + 2} - 1} = \tau) = 0$.

Otherwise, for each $i = 0,\ldots,N_t - 1$, the probability of the event
\begin{quote}
\begin{itemize}
\item[$E_{t,i}$:] $\omega_{N_{t,i}}^{N_{t,i + 1}}$ is not equal to any substring of $\tau$
\end{itemize}
\end{quote}
is given by
\[
\prob(E_{t,i} \given \omega_{2^{2t + 1}}^{2^{2t + 2} - 1} = \tau) = 1 - \sum_{\substack{\eta\in E^{\ell_t} \\ \text{substring of }\tau}}e^{-\delta\pseud(\eta)}
\]
and is therefore bounded above by
\[
1 - (2^{2t + 1} - \ell_t - \alpha)e^{-\delta(\ell_t + \alpha)}.
\]
By independence, it follows that the probability that $E_{t,i}$ holds for all $i = 0,\ldots,N_t - 1$ is bounded above by
\begin{equation}
\label{severalbound}
\left(1 - (2^{2t + 1} - \ell_t - \alpha)e^{-\delta(\ell_t + \alpha)}\right)^{N_t}.
\end{equation}
On the other hand, if $E_t$ holds, it is evident that $E_{t,i}$ holds for all $i = 0,\ldots,N_t - 1$. Thus the probability of $E_t$ given $\omega_{2^{2t + 1}}^{2^{2t + 2} - 1} = \tau$ is bounded above by (\ref{severalbound}). Since this conclusion holds for all $\tau\in E^{2^{2t}}$, it follows that the unconditional probability of $E_t$ is bounded above by (\ref{severalbound}).

As noted above, if (\ref{badlyapproximable4}) holds for every repeat $\eta$, then $E_t$ holds for all $t$. Since the sequence $(E_t)_t$ is independent, we have
\begin{align*}
\prob\left(\bigcap_{t\in\N} E_t\right) &\leq \prod_{t\in\N}\left(1 - (2^{2t + 1} - \ell_t - \alpha)e^{-\delta(\ell_t + \alpha)}\right)^{N_t}\\
&\leq \prod_{t\in\N}\exp\left(-N_t(2^{2t + 1} - \ell_t - \alpha)e^{-\delta(\ell_t + \alpha)}\right)\\
&= \exp\left(-\sum_{t\in\N}N_t(2^{2t + 1} - \ell_t - \alpha)e^{-\delta(\ell_t + \alpha)}\right).
\end{align*}
In particular, if the sum
\begin{equation}
\label{ellt}
\sum_{t\in\N}N_t(2^{2t + 1} - \ell_t - \alpha)e^{-\delta(\ell_t + \alpha)} \asymp \sum_{t\in\N}\frac{2^{4t}}{\ell_t}e^{-\delta\ell_t}
\end{equation}
diverges, then the probability that (\ref{badlyapproximable4}) holds for every repeat $\eta$ is zero. Since the divergence of the sum will be shown to be independent of $K$, it follows that if the sum diverges, then $\mu$-almost every point $x$ is intrinsically well approximable with respect to $\psi$.

Write $\alpha \leq 2^{2r - 2}$ for some $r\in\N$. Then
\begin{align*}
\sum_{t\in\N}\frac{2^{4t}}{\ell_t} e^{-\delta\ell_t}
&\geq \sum_{t\in\N}\frac{2^{4t}}{K + \Psi(2^{2t + 2r})} e^{-\delta(K + \Psi(2^{2t + 2r}))}\\
&\geq \frac{1}{2^{4r + 2}}\sum_{t\geq r}2^{2t}\frac{2^{2t + 2}}{K + \Psi(2^{2t})} e^{-\delta(K + \Psi(2^{2t}))}\\
&\geq \frac{1}{3}\frac{1}{2^{4r + 2}}\sum_{t\geq r}\sum_{n = 2^{2t}}^{2^{2t + 2} - 1}\frac{n}{K + \Psi(n)} e^{-\delta(K + \Psi(n))}\\
&\asymp \sum_{n = 0}^\infty\frac{\lfloor e^{n + 1}\rfloor - \lfloor e^n\rfloor}{e^n}\frac{n + 1}{\Psi(n)} e^{-\delta\Psi(n)}\\
&\geq \sum_{n = 0}^\infty\sum_{q = \lfloor e^n\rfloor}^{\lfloor e^{n + 1}\rfloor - 1}\frac{\log(q)}{q\Psi(\log(q))} e^{-\delta\Psi(\log(q))}\\
&= \sum_{q = 1}^\infty \frac{\log(q)}{q\Psi(\log(q))} e^{-\delta\Psi(\log(q))}\\
&= \sum_{q = 1}^\infty \frac{\log(q)\psi(q)^\delta}{q\log(\psi(q))}
\end{align*}
so if (\ref{qdeltalnqlnlnq}) diverges then (\ref{ellt}) diverges as well.
\end{enumerate}
\end{proof}

\section{Optimality of the bound}
\label{optimality}
In this section, we will restrict ourselves to the case where $J$ is the Cantor set $C$.

We begin by recalling the following conjecture and proposition from \cite{BFR}:
\begin{conjecture}[{\cite[Conjecture 3.3]{BFR}}]
\label{conjectureBFR}
If
\[
S_n := \{p/q\in C:\gcd(p,q) = 1,\;3^{n - 1} \leq q < 3^n\}
\]
then for all $\varepsilon_1 > 0$ we have
\[
\#(S_n) = O(2^{n(1 + \varepsilon_1)}).
\]
\end{conjecture}

\begin{proposition*}[{\cite[Corollary 3.4]{BFR}}]
Conjecture \ref{conjectureBFR} implies that $\mu(\VWA_C) = 0$, where
\[
\VWA_C := \{x\in C:\exists \varepsilon > 0\;\;\exists^\infty p/q\in C \;\; |x - p/q| \leq q^{-(1 + \varepsilon)}\}.
\]
\end{proposition*}
As mentioned in the Introduction, we cannot prove Conjecture \ref{conjectureBFR} at this time, but we will reduce it to a simpler conjecture which a heuristic argument suggests is true.

\begin{definition}
Suppose that $p/q$ is a rational number. The \emph{period} of $p/q$ is the period of the ternary expansion of $p/q$, and will be denoted $P(p/q)$.
\end{definition}

\begin{theorem}
\label{theoremSnK}
For every $K < \infty$, if
\[
S_n^{(K)} := \{p/q\in C:\gcd(p,q) = 1,\;3^{n - 1} \leq q < 3^n,\textup{ and }P(p/q) \leq K \log(q)\}
\]
then for all $\varepsilon_1 > 0$ we have
\[
\#(S_n^{(K)}) = O(2^{n(1 + \varepsilon_1)}).
\]
\end{theorem}

We postpone the proof of Theorem \ref{theoremSnK} to the end of this section and proceed to state the following immediate corollary:

\begin{corollary}
\label{corollaryoftheoremSnK}
The following conjecture implies Conjecture \ref{conjectureBFR}, and thus that $\mu(\VWA_C) = 0$:
\vspace{-0.27 in}
\begin{quote}
\begin{conjecture}
\label{conjectureSnK}
There exists $K < \infty$ such that
\[
\#(S_n\butnot S_n^{(K)}) = O(2^{n(1 + \varepsilon_1)}).
\]
\end{conjecture}
\end{quote}
\end{corollary}

We will offer a heuristic argument in support of Conjecture \ref{conjectureSnK}. This argument will in fact support the following much stronger conjecture:
\begin{conjecture}
\label{conjectureheuristic}
For all $K > 2/\log(3/2)$, we have $S_n^{(K)} = S_n$ for all $n$ sufficiently large. In particular
\[
\#(S_n\butnot S_n^{(K)}) = o(1).
\]
\end{conjecture}

\begin{proof}[Heuristic argument for Conjecture \ref{conjectureheuristic}]
It is easily verified that Conjecture \ref{conjectureheuristic} is equivalent to the inequality
\begin{equation}
\label{329}
\limsup_{\substack{p,q \\ p/q\in C \\ q\rightarrow\infty}}\frac{P(p/q)}{\log(q)} \leq \frac{2}{\log(3/2)}.
\end{equation}
Our method is to estimate reality using a probabilistic model, and then show that (\ref{329}) holds with probability one.

We will not specify our model exactly, but we will assume that it has the following property:
\begin{quote}
For each $p/q\in\Q$, the digits of $p/q$ are independent and identically distributed until they start repeating.
\end{quote}
We do not assume any independence of the digits of $p/q$ from the digits of any other rational, nor any estimate of the distribution of the periods.

Based on this assumption, if $p/q\in[0,1]$ is fixed then the probability that $p/q\in C$ given that $P(p/q) = m$ is $(2/3)^m$. It follows from standard probability theory that
\[
\prob(p/q\in C \text{ and }P(p/q)\geq m) \leq (2/3)^m.
\]
Fix $\varepsilon > 0$. We have
\[
\prob(p/q\in C \text{ and }P(p/q)\geq (2 + \varepsilon)\log_{3/2}(q)) \leq q^{-(2 + \varepsilon)}.
\]
For each $Q$, the probability that there exist $p,q$ with
\begin{align*}
q &\geq Q\\
p/q &\in C\\
P(p/q) &\geq (2 + \varepsilon)\log_{3/2}(q)
\end{align*}
is at most
\[
\sum_{\substack{p/q\in[0,1] \\ q\geq Q}}q^{-(2 + \varepsilon)} = \sum_{q\geq Q}q^{-(1 + \varepsilon)} \tendsto Q 0.
\]
Thus with probability one, there exists $Q$ such that for all $p,q$ with $q\geq Q$ and $p/q\in C$, we have $P(p/q)\leq \log_{3/2}(q)(2 + \varepsilon)$. Rearranging yields (\ref{329}).
\end{proof}
\begin{remark}
The weakest part of this heuristic argument is the fact that the randomness is not open to a statistical interpretation. We are not saying ``If you pick a rational at random, this should happen'' but rather ``If you pick a random mathematical universe, then this should happen'' (which of course makes no sense as a logical statement). In fact, the former statement would be insufficient to support Conjecture \ref{conjectureheuristic} (or even Conjecture \ref{conjectureSnK}), since we need that the size of the set of exceptions in proportion to the set of all rationals in a given range tends to zero exponentially fast.
\end{remark}
\begin{proof}[Proof of Theorem \ref{theoremSnK}]
Let $K_4 = K\log(3)$. We have
\[
S_n^{(K)} \subset \bigcup_{m = 1}^{K_4 n}\{p/q\in C:\gcd(p,q) = 1,\;q < 3^n,\text{ and }P(p/q) = m\}.
\]
For each $q < 3^n$, we have
\[
\#\{p = 0,\ldots,q: p/q\in C\} \leq K_5 2^n
\]
by the fractal pigeonhole principle. Thus
\[
\#(S_n^{(K)}) \leq K_5 2^n\sum_{m = 1}^{K_4 n}\#\{q < 3^n: \exists p \;\; \gcd(p,q) = 1,\;P(p/q) = m\}.
\]
Fix $q\in\N$, and suppose that there exists $p$ with $\gcd(p,q) = 1$ and $P(p/q) = m$. Write $q = 3^r \w{q}$ where $3$ does not divide $\w{q}$. Then $\gcd(p,\w{q}) = 1$ and $P(p,\w{q}) = m$. Furthermore the ternary expansion of $p/\w{q}$ is (immediately) periodic. A simple calculation shows that $p/\w{q} = i/(3^m - 1)$ for some $i = 0,\ldots,3^m - 1$. Since $p/\w{q}$ is in reduced form, this implies that $\w{q}$ divides $3^m - 1$. To summarize:
\begin{align*}
\#\{q < 3^n: \exists p \;\; \gcd(p,q) = 1, P(p/q) = m\} &\leq \#\{(r,\w{q}):0\leq r < n,\;\w{q}\text{ divides }3^m - 1\}\\
&= n \tau(3^m - 1),
\end{align*}
where $\tau$ is the number-of-divisors function.

The following result concerning the number-of-divisors function was proven by Ramanujan \cite{Ramanujan}:
\[
\limsup_{N\to\infty}\frac{\log(\tau(N))}{\log(N)/\log\log(N)} = \log(2).
\]
Thus for every $\varepsilon > 0$, we have
\[
\tau(N) \leq N^{(\log(2) + \varepsilon)/\log\log(N)}
\]
for all $N$ sufficiently large. In particular, if we fix $\varepsilon_2 > 0$ to be determined, then
\[
\tau(N) \leq N^{\varepsilon_2}
\]
for all $N$ sufficiently large. Let $K_{6,\varepsilon_2}$ be large enough so that
\[
\tau(N) \leq K_{6,\varepsilon_2} N^{\varepsilon_2}
\]
for all $N\in\N$.

Combining our several equations yields
\[
\#(S_n^{(K)}) \leq K_5 K_{6,\varepsilon_2} n 2^n\sum_{m = 1}^{K_4 n}(3^m - 1)^{\varepsilon_2} \asymp n 2^n 3^{n K_4 \varepsilon_2} \lesssim 2^{n(1 + \varepsilon_1)}
\]
if $\varepsilon_2$ is chosen small enough so that $3^{K_4 \varepsilon_2} < 2^{\varepsilon_1}$.

\end{proof}

\bibliographystyle{amsplain}

\bibliography{bibliography}

\providecommand{\bysame}{\leavevmode\hbox to3em{\hrulefill}\thinspace}
\providecommand{\MR}{\relax\ifhmode\unskip\space\fi MR }
\providecommand{\MRhref}[2]{%
  \href{http://www.ams.org/mathscinet-getitem?mr=#1}{#2}
}
\providecommand{\href}[2]{#2}
\begin{thebibliography}{10}

\bibitem{BFR}
R.~Broderick, L.~Fishman, and A.~Reich, \emph{Intrinsic approximation on
  {C}antor-like sets, a problem of {M}ahler}, Moscow Journal of Combinatorics
  and Number Theory \textbf{1} (2011), 291--300.

\bibitem{Bugeaud2}
Y.~Bugeaud, \emph{{D}iophantine approximation and {C}antor sets}, Mathematische
  Annalen \textbf{341} (2008), 677--684.

\bibitem{Fishman}
L.~Fishman, \emph{{S}chmidt's game on fractals}, Israel J. Math. \textbf{171}
  (2009), no. 1, 77--92.

\bibitem{FishmanSimmons2}
L.~Fishman and D.~S. Simmons, \emph{Extrinsic {D}iophantine approximation on
  manifolds and fractals}, \url{http://arxiv.org/abs/1406.0785}, 2014,
  preprint.

\bibitem{Hutchinson}
J.~E. Hutchinson, \emph{Fractals and self-similarity}, Indiana Univ. Math. J.
  \textbf{30} (1981), no. 5, 713--747.

\bibitem{KLW}
D.~Y. Kleinbock, E.~Lindenstrauss, and B.~Weiss, \emph{On fractal measures and
  {D}iophantine approximation}, Selecta Math. \textbf{10} (2004), 479--523.

\bibitem{KleinbockWeiss1}
D.~Y. Kleinbock and B.~Weiss, \emph{Badly approximable vectors on fractals},
  Israel J. Math. \textbf{149} (2005), 137--170.

\bibitem{KTV}
S.~Kristensen, R.~Thorn, and S.~L. Velani, \emph{{D}iophantine approximation
  and badly approximable sets}, Advances in Math. \textbf{203} (2006),
  132--169.

\bibitem{LSV}
J.~Levesley, C.~Salp, and S.~L. Velani, \emph{On a problem of {K}. {M}ahler:
  {D}iophantine approximation and {C}antor sets}, Math. Ann. \textbf{338}
  (2007), 97--118.

\bibitem{Mahler}
K.~Mahler, \emph{Some suggestions for further research}, Bull. Aust. Math. Soc.
  \textbf{29} (1984), 101--108.

\bibitem{Mattila}
P.~Mattila, \emph{Geometry of sets and measures in {E}uclidean spaces: Fractals
  and rectifiability}, Cambridge Studies in Advanced Mathematics, 44, Cambridge
  University Press, Cambridge, 1995.

\bibitem{Ramanujan}
S.~A. Ramanujan, \emph{On the number of divisors of a number ({J}ournal of the
  {I}ndian {M}athematical {S}ociety 7 (1915), 131-133)}, Collected papers of
  Srinivasa Ramanujan, AMS Chelsea Publ., Providence, RI, 2000, pp.~44--46.

\bibitem{Schmidt3}
W.~M. Schmidt, \emph{{D}iophantine approximation}, Lecture Notes in
  Mathematics, vol. 785, Springer-Verlag, Berlin, 1980.

\bibitem{Weiss}
B.~Weiss, \emph{Almost no points on a {C}antor set are very well approximable},
  R. Soc. Lond. Proc. Ser. A Math. Phys. Eng. Sci. \textbf{457} (2001), no.
  2008, 949--952.

\end{thebibliography}

\end{document}